\documentclass[12pt]{amsart}
\usepackage{amssymb}
\usepackage{amsmath}
\usepackage{cases}
\usepackage{}
\usepackage{mathrsfs}
\usepackage{amsfonts}
\usepackage{txfonts}
\usepackage{pifont}
\usepackage{cite}

\newtheorem{theorem}{Theorem}[section]
\newtheorem{lemma}[theorem]{Lemma}
\newtheorem{corollary}[theorem]{Corollary}

\theoremstyle{definition}

\theoremstyle{remark}
\newtheorem{remark}[theorem]{Remark}

\newcommand{\ind}{\textup{ind}}
\newcommand{\rank}{\textup{rank}}

\numberwithin{equation}{section}

\begin{document}

\title{Drazin inverses of the sum of four matrices and its applications }
\author{Daochang Zhang}
\address{College of Sciences, Northeast Electric Power University, Jilin, P.R. China.}
\email{daochangzhang@126.com}
\thanks{The first author is supported by the NSFC (No. 11371165; No. 11401249)
and the Scientific and Technological Development Program Foundation of Jilin
Province, China (No. 20170520052JH). The second author is supported by
the Ministry of Education and Science, Republic of Serbia (
No. 174007).}
\author{Dijana Mosi\' c}
\address{Faculty of Sciences and Mathematics, University of Ni\v s, P.O.
Box 224, 18000 Ni\v s, Serbia.} \email{dijana@pmf.ni.ac.rs}
\author{Li Guo}
\address{1. Department of Mathematics, Southeast University, Nanjing  210096, China.
2. School of Mathematics and  Statistics, Beihua University, Jilin  132013, China.} \email{guomingli95@163.com}


\date{}



\begin{abstract}
Our aim is to establish relations between Drazin inverses of the pesudo-block matrix $(P,Q,R,S)$  and
the block matrix $\left[\begin{smallmatrix}
   P&R\\S& Q
\end{smallmatrix}\right]$, where $R^2=S^2=0$. Based on the relations, we give representations for Drazin inverses of the sum $P+Q+R+S$ under weaker restrictions.
As its applications,
several expressions for Drazin inverses of a $2\times 2$ block matrix are presented under some assumptions. Our results generalize several results in the literature.
\end{abstract}

\maketitle

\section{introduction}
The Drazin inverse of a complex square matrix $A$ is the unique matrix $A^d$ such that
\begin{equation*}\label{defin1}
  AA^d=A^dA,~~~ A^dAA^d=A^d,~~~ A^{k}=A^{k+1}A^d,
\end{equation*}
 where $k$ is the smallest non-negative integer such that $\rank(A^{k}) = \rank(A^{k+1})$, called index of $A$ and denoted by $\ind(A)$. If $\ind(A)=1$, then $A^{d}$ is called the group inverse of $A$ and denoted by $A^{\#}$.
The Drazin inverse is a generalization of inverses and group inverses of matrices. There are widespread applications of Drazin inverses in various fields,
such as differential equations, control theory, Markov chains, iterative methods and so on
(see \cite{Israel2003,Eiermann1988,Hanke1994,Hartwig1981,Kirkland2004,Li2003,Meyer1980,Szyld1994,Wei2005,Wei2006}).

There have been many papers considering the problem of finding a
formula for the Drazin inverse of a $2 \times 2$ matrix in terms
of its various blocks, where the blocks on the diagonal are
required to be square matrices(see \cite{HLW,Meyer(1977)}).
This problem was proposed by Campbell and Meyer \cite{Campbell1991}, but it is
still an open problem to find an explicit formula for the Drazin
inverse of a block matrix without any restrictions upon the
blocks.

Let $\mathbb{C}^{m\times n}$ denote the set of all ${m\times n}$ complex matrices.
For $N\in \mathbb{C}^{n\times n}$, if there exist $P, Q, R,S \in \mathbb{C}^{n\times n}$ such that $N=P+Q+R+S$ and
$$PQ=QP=0,~PS=SQ=QR=RP=0,~R^d=S^d=0,$$
then the quadruple $(P,Q,R,S)$ is called a pseudo-block decomposition of $N$, and $N$ is called a pseudo-block matrix corresponding to $(P,Q,R,S)$ (see \cite{Chen2008}). In this case, we will simply say that $N=(P,Q,R,S)$ is a pseudo-block matrix.

In Section 2,
we establish relations between $\left[\begin{array}{cc}
        P & R \\
        S & Q
      \end{array}\right]^d$ and $(P+Q+R+S)^d$ for a pseudo-block decomposition $(P,Q,R,S)$ with $R^2=S^2=0$.
In Section 3, we give representations for $N^d$ under weaker restrictions and generalize some results of \cite{Chen2008}.
Furthermore, in Section 4, considering some matrix decompositions,
we apply $N^d$ to get some new expressions for the Drazin inverse of a $2\times 2$ block matrix under some conditions,
which extend several results of \cite{HLW,MDijcmEx,MDjrepgdAMC1,MDjamc2} and recover some result of \cite{DjS,MZCannalsFunc}
for the case of Drazin inverses.

Throughout this paper,
denote by $I$ the identity matrix of proper size, write $A^\pi$ for $I-AA^d$, and
let
$$M=\left[\begin{array}{cc}
                   A & B \\
                   C & D
                 \end{array}\right]\quad
                 {\rm and}\quad Z=D-CA^dB,$$
where $A\in\mathbb{C}^{m\times m}$, $B\in\mathbb{C}^{m\times n}$, $C\in\mathbb{C}^{n\times m}$ and $D\in\mathbb{C}^{n\times n}$.

For notational convenience, we define a sum to be 0, whenever its lower limit is bigger than its upper limit.
\section{relations between Drazin inverses of the pesudo-block matrix and the block matrix}

\begin{lemma}\label{mathematical induction}
Let $T=\left[\begin{array}{cc}
        P & R \\
        S & Q
      \end{array}\right]$ for $P,Q,R,S\in \mathbb{C}^{n\times n}$  such that $N=(P,Q,R,S)$ is a pseudo-block matrix with $R^2=S^2=0$.
If $T^r=\begin{bmatrix}
  A_1&A_3\\A_4&A_2
\end{bmatrix} $
for $A_1,A_2,A_3,A_4\in\mathbb{C}^{n\times n}$ and any positive integer $r$,
then $N^r=A_1+A_2+A_3+A_4$ and
$(A_1,A_2,A_3,A_4)$ is a pseudo-block matrix with $A_1Q=A_2P=A_1S=A_4Q=A_2R=A_3P=A_3^{2}=A_4^{2}=0$.
\end{lemma}
\begin{proof}
 We will prove the lemma by induction on $r$. The lemma is obvious when $r=1$.

Now, assume that the lemma holds for $r=k-1$ and let us prove that it holds for $r=k$. Let $$
T^{k-1}=\begin{bmatrix}
B_1&B_3\\
B_4&B_2
\end{bmatrix}
$$
for $B_1,B_2,B_3,B_4\in\mathbb{C}^{n\times n}$.  Then $(B_1,B_2,B_3,B_4)$ is a pseudo-block matrix such that $N^{k-1}=B_1+B_2+B_3+B_4$ and $B_1Q=B_2P=B_1S=B_4Q=B_2R=B_3P=B_3^{2}=B_4^{2}=0$.
Suppose that $T^{k}=\begin{bmatrix}
A_1&A_3\\
A_4&A_2
\end{bmatrix}$.
Since
$$
T^{k}=\begin{bmatrix}
B_1&B_3\\
B_4&B_2
\end{bmatrix}
\begin{bmatrix}
P&R\\
S&Q
\end{bmatrix}
=\begin{bmatrix}
B_1P+B_3S&B_1R+B_3Q\\
B_4P+B_2S&B_4R+B_2Q
\end{bmatrix},
$$
we have $A_1=B_1P+B_3S$, $A_2=B_4R+B_2Q$, $A_3=B_1R+B_3Q$,   and $A_4=B_4P+B_2S$.
Thus
\begin{align*}
N^k&=(B_1+B_2+B_3+B_4)(P+Q+R+S)\\
&=B_1P+B_1R+B_2Q+B_2S+B_3Q+B_3S+B_4P+B_4R\\
&=A_1+A_2+A_3+A_4.
   \end{align*}
Now it is a routine matter to check that $(A_1,A_2,A_3,A_4)$ is a pseudo-block matrix and $A_1Q=A_2P=A_1S=A_4Q=A_2R=A_3P=A_3^{2}=A_4^{2}=0$.
\end{proof}
\begin{lemma}\label{TrTs}
Let $T=\left[\begin{array}{cc}
        P & R \\
        S & Q
      \end{array}\right]$ for $P,Q,R,S\in \mathbb{C}^{n\times n}$  such that $N=(P,Q,R,S)$ is a pseudo-block matrix with $R^2=S^2=0$.
If $T^i=\begin{bmatrix}
  A^{(i)}_1&A^{(i)}_3\\A^{(i)}_4&A^{(i)}_2
\end{bmatrix} $
for $A^{(i)}_1,A^{(i)}_2,A^{(i)}_3,A^{(i)}_4\in\mathbb{C}^{n\times n}$ and for any positive integer $i$,
then
$$A^{(r)}_1A^{(s)}_2=A^{(r)}_2A^{(s)}_1=A^{(r)}_1A^{(s)}_4
=A^{(r)}_4A^{(s)}_2=A^{(r)}_2A^{(s)}_3=A^{(r)}_3A^{(s)}_1=0$$ for any positive integers $r$ and $s$.
\end{lemma}
\begin{proof}Since $T^s=TT^{s-1}$, a calculation yields
$$A^{(s)}_1=PA^{(s-1)}_1+RA^{(s-1)}_4,
~~A^{(s)}_2=SA^{(s-1)}_3+QA^{(s-1)}_2,~~
    A^{(s)}_3=PA^{(s-1)}_3+RA^{(s-1)}_2,~~
    A^{(s)}_4=SA^{(s-1)}_1+QA^{(s-1)}_4.
    $$
Now by Lemma \ref{mathematical induction} it is a routine matter to verify that   $$A^{(r)}_1A^{(s)}_2=A^{(r)}_2A^{(s)}_1=A^{(r)}_1A^{(s)}_4
=A^{(r)}_4A^{(s)}_2=A^{(r)}_2A^{(s)}_3=A^{(r)}_3A^{(s)}_1=0$$
as desired.
\end{proof}
\begin{lemma}\label{f(T),f(N)}
Let $T=\left[\begin{array}{cc}
        P & R \\
        S & Q
      \end{array}\right]$ for $P,Q,R,S\in \mathbb{C}^{n\times n}$  such that $N=(P,Q,R,S)$ is a pseudo-block matrix with $R^2=S^2=0$.
If $f(T)=\begin{bmatrix}
  A_1&A_3\\A_4&A_2
\end{bmatrix} $
for $A_1,A_2,A_3,A_4\in\mathbb{C}^{n\times n}$ and for a polynomial $f(x)\in \mathbb{C}[x]$,
then $f(N)=A_1+A_2+A_3+A_4$ and
$(A_1,A_2,A_3,A_4)$ is a pseudo-block matrix with $A_3^{2}=A_4^{2}=0$.
\end{lemma}
\begin{proof}It follows from  Lemma \ref{mathematical induction} and Lemma \ref{TrTs}.
\end{proof}

Now we can give our first main result.
\begin{theorem}
Let $T=\left[\begin{array}{cc}
        P & R \\
        S & Q
      \end{array}\right]$ for $P,Q,R,S\in \mathbb{C}^{n\times n}$  such that $N=(P,Q,R,S)$ is a pseudo-block matrix with $R^2=S^2=0$.
If $T^d=\begin{bmatrix}
  A_1&A_3\\A_4&A_2
\end{bmatrix} $
for $A_1,A_2,A_3,A_4\in\mathbb{C}^{n\times n}$,
then $N^d=A_1+A_2+A_3+A_4$ and
$(A_1,A_2,A_3,A_4)$ is a pseudo-block matrix with $A_3^{2}=A_4^{2}=0$.
\end{theorem}
\begin{proof}
Let $\mathbb{C}[N]$ and $\mathbb{C}[T]$ denote the subalgebra generated by $N$ and $T$.
Consider the following homomorphisms
\begin{align*}
\phi: \mathbb{C}[x]&\rightarrow\mathbb{C}[T], \quad f(x)\mapsto f(T);\\
\psi: \mathbb{C}[x]&\rightarrow\mathbb{C}[N],\quad f(x)\mapsto f(N).
\end{align*}
By Lemma \ref{f(T),f(N)} we get $\ker\phi\subseteq\ker\psi $.
Thus there exists a homomorphism
$\tau:\mathbb{C}[T]\rightarrow\mathbb{C}[N]$ such that
$\tau\phi=\psi.$ Clearly, $\tau(f(T))=f(N)$ for all $f(x)\in\mathbb{C}[x]$.
By \cite[Theorem 7.5.1]{Campbell1991}, there exists a polynomial
$p(x)$ such that $T^d=p(T)$, and so $T^d\in \mathbb{C}[T]$. Similarly,
$N^d\in\mathbb{C}[N]$. Since the Drazin inverse is preserved by homomorphisms, we have that $N^d=\tau(T)^d=\tau(T^d)=A_1+A_2+A_3+A_4$.
\end{proof}


\section{Drazin inverses of the sum $P+Q+R+S$}
\begin{lemma}\label{(eaf)^{d}=ea^{d}f}
Let $A\in \mathbb{C}^{n\times n}$ and $e,f$ be idempotents of $\mathbb{C}^{n\times n}$. If $(eAf)^{i}=eA^{i}f$ for all $i\geq1$, then
 \begin{equation*}
 (eAf)^{d}=eA^{d}f.
 \end{equation*}
\end{lemma}
\begin{proof}
By \cite[Theorem 7.5.1]{Campbell1991}, $A^d=p(A)$ for a polynomial $p(x)\in\mathbb{C}[x]$.
Since $A^d=(A^d)^2A$, we may assume that the constant term of $p$ is zero. Let $p(x)=\sum_{i=1}^n a_{i}x^{i}$,  where $a_i\in\mathbb{C}$. We have
$$A^d=a_nA^n+a_{n-1}A^{n-1}+\dots+a_1A.$$
Since $(eAf)^{i}=eA^{i}f$ for all $i\geq1$, we have
$$eA^d f=a_n(eAf)^n+a_{n-1}(eAf)^{n-1}+\dots+a_1eAf.$$
Then
\begin{eqnarray*}
 eA^i f\cdot eA^d f&=&(eAf)^i\cdot[a_n(eAf)^n+a_{n-1}(eAf)^{n-1}+\dots+a_1eAf]\\
&=&a_n(eAf)^{n+i}+a_{n-1}(eAf)^{n+i-1}+\dots+a_1(eAf)^i\\
&=&e(a_nA^{n+i}+a_{n-1}A^{n+i-1}+\dots+a_1A^{i})f\\&=&eA^{i}A^d f,
\end{eqnarray*}
for all $i\geq1$.  A calculation gives
 \begin{align*}
   &eAf\cdot eA^{d}f=eAA^{d}f=eA^{d}f\cdot eAf,\\
   &eAf\cdot(eA^{d}f)^{2}=eAf\cdot e(A^{d})^2f=eA(A^{d})^2f=eA^{d}f,\\
   &(eAf)^{k+1}\cdot eA^{d}f=eA^{k+1}f\cdot eA^{d}f=eA^{k+1}A^{d}f=eA^{k}f=(eAf)^{k},
 \end{align*}
where $k=ind(A)$.
Thus, $(eAf)^{d}=eA^{d}f$.
\end{proof}

\begin{lemma}\label{P+Q}\cite[Theorem 2.1]{Hartwig(2001)}
Let $P$ and $Q$ be ${n\times n}$ matrices. If $PQ=0$, then
$$
  (P+Q)^{d}=Q^{\pi}\sum_{i=0}^{t-1}Q^{i}(P^{d})^{i+1}+\sum_{i=0}^{s-1}(Q^{d})^{i+1}P^{i}P^{\pi},
$$
where $s= \ind(P)$ and $t=\ind(Q)$.
\end{lemma}

Let $\mathbb{A}$ denote a complex unital algebra, and let
$M_2(\mathbb{A})$ be the $2\times2$ matrix algebra over $\mathbb{A}$.
Given an idempotent $e$ in $\mathbb{A}$, we consider a mapping $\sigma$ from $\mathbb{A}$ to $M_2(\mathbb{A},e)$
and the set
\begin{equation*}\label{alis}
M_2(\mathbb{A},e)=\begin{pmatrix}
  e\mathbb{A}e&e\mathbb{A}(1-e)\\
  (1-e)\mathbb{A}e&(1-e)\mathbb{A}(1-e)
\end{pmatrix}\subset M_2(\mathbb{A}).
\end{equation*}

\begin{lemma}\label{M=P+Q+R+S:equivalentM,N}
Let $e$ be an idempotent of $\mathbb{A}$. For any $a\in\mathbb{A}$ let
$$\sigma(a)
=\begin{pmatrix}
  eae&ea(1-e)\\
  (1-e)ae&(1-e)a(1-e)
\end{pmatrix}
\in M_2(\mathbb{A},e).$$
Then the mapping $\sigma$ is an algebra isomorphism from
$\mathbb{A}$ to $M_2(\mathbb{A},e)$ such that
\begin{enumerate}
\item
$(\sigma(a))^d=\sigma(a^d)$;\\
\item if $(\sigma(a))^d=\begin{pmatrix}
     \alpha & \beta \\
     \gamma & \delta
   \end{pmatrix}$, then $a^{d}=\alpha+\beta+\gamma+\delta$.
\end{enumerate}
\end{lemma}
\begin{proof}
It is clear that the mapping $\sigma$ is an algebra isomorphism from $\mathbb{A}$ to $M_2(\mathbb{A},e)$.
Then the property (1) follows, and the property (2) follows from
$$\sigma(a^d)=(\sigma(a))^d
 =\begin{pmatrix}
     \alpha & \beta \\
     \gamma & \delta
   \end{pmatrix}$$
as desired.
\end{proof}

In what follows, let $\sigma$ be defined as in Lemma \ref{M=P+Q+R+S:equivalentM,N}.
\begin{lemma} Let $P,Q,R$ and $S$ $\in\mathbb{C}^{n\times n}$ and let $N=P+Q+R+S$.
If $PQ=QP=PS=SQ=QR=RP=0$,
then
$$\sigma(N)
=\begin{pmatrix}
  P^2P^{d} & PP^{d}R \\
  SPP^{d} & Q+P^{\pi}R+(S+P)P^{\pi}
\end{pmatrix}.$$
\end{lemma}
\begin{proof}
Note that mapping $\sigma$ is an algebra isomorphism from $\mathbb{A}$ to $M_2(\mathbb{A},e)$.
Then
$$\sigma(N)=\begin{pmatrix}
  eNe&eN(1-e)\\
  (1-e)Ne&(1-e)N(1-e)
  \end{pmatrix}.$$
Let $P^{e}=PP^d$. Since $P^{e}$ is idempotent, a calculation gives
 \begin{align*}
   &P^{e}NP^{e}=P^{e}P,\\
   &P^{e}NP^{\pi}=P^{e}R,\\
   &P^{\pi}NP^{e}=SP^{e},\\
   &P^{\pi}NP^{\pi}=Q+P^{\pi}R+(S+P)P^{\pi}
 \end{align*}
as desired.
\end{proof}

In the above assumption, we consider the Drazin inverse of the lower right element of $\sigma(N)$.

\begin{lemma}\label{SP=SR=0,2-2}
Let $P,Q,R$ and $S$ $\in\mathbb{C}^{n\times n}$. If $PQ=QP=PS=SQ=QR=RP=0$ and $SP=SR=0$, then
  \begin{align*}
    &\big((S+Q)+P^{\pi}(R+P)\big)^{d}\\
   =&P^\pi(\sum_{j=0}^{n-1}(\sum_{k=1}^{j}P^kR^{j-k}+R^{\pi}R^j)-\sum_{i=0}^{m-1}\sum_{j=0}^{n-1}P^{i+1}(R^d)^{i+1}R^j)\\
&\quad\times(Q^\pi\sum_{h=0}^{l-1}Q^h(S^d)^{h+j+1}-\sum_{k=1}^{j}(Q^d)^k(S^d)^{j+1-k}+\sum_{h=0}^{l-1}(Q^d)^{h+j+1}S^{h}S^\pi)\\
&+P^\pi\sum_{j=0}^{n-1}\sum_{i=0}^{m-1}P^i(R^d)^{i+j+1}
(\sum_{k=0}^{j}Q^kS^{j-k}S^\pi-Q^\pi\sum_{h=0}^{l-1}Q^{h+1+j}(S^d)^{h+1}-\sum_{h=0}^{l-1}Q^{j+1}(Q^d)^{h+1}S^{h}S^\pi),
 \end{align*}
where
 \begin{align*}
 &\max\{\ind(S),\ind(Q)\}\leq l\leq \ind(S)+\ind(Q),\\
 &\max\{\ind(R),\ind(P)\}\leq m\leq \ind(R)+\ind(P),\\
 &\max\{\ind(S+Q),\ind(P^{\pi}(R+P))\}\leq n \leq \ind(S+Q)+\ind(P^{\pi}(R+P)).
 \end{align*}
\end{lemma}
\begin{proof}
  We first note that $(S+Q)\cdot P^{\pi}(R+P)=0$.
By  Lemma \ref{P+Q}, we have
\begin{multline*}
  \big((S+Q)+P^{\pi}(R+P)\big)^{d}
  =\big(P^{\pi}(R+P)\big)^{\pi}\sum_{j=0}^{n-1}(P^{\pi}(R+P))^{j}\big((S+Q)^d\big)^{j+1}\\
  +\sum_{j=0}^{n-1}((P^{\pi}(R+P))^{d})^{j+1}(S+Q)^{j}(S+Q)^{\pi},
\end{multline*}
where
$$\max\{\ind(S+Q),\ind(P^{\pi}(R+P))\}\leq n \leq \ind(S+Q)+\ind(P^{\pi}(R+P)).$$
Since $P^{\pi}(R+P)P^{\pi}=P^{\pi}(R+P)$, we have $(P^{\pi}(R+P))^i=P^{\pi}(R+P)^i$ for $i\geq1$.
Then, by Lemma \ref{(eaf)^{d}=ea^{d}f},
\begin{align}\label{Ppi(R+P)}
(P^{\pi}(R+P))^d=P^{\pi}(R+P)^d\quad\text{and}\quad(P^{\pi}(R+P))^{\pi}=1-P^{\pi}(R+P)(R+P)^{d}.
\end{align}
The hypothesis $SQ=0$  gives
\begin{align}\label{(S+Q)^{d}}
  (S+Q)^{d}=Q^{\pi}\sum_{h=0}^{l-1}Q^{h}(S^{d})^{h+1}+\sum_{h=0}^{l-1}(Q^{d})^{h+1}S^{h}S^{\pi},
\end{align}
where $\max\{\ind(S),\ind(Q)\}\leq l\leq \ind(S)+\ind(Q)$.
Because $RP=0$, we get
\begin{align}\label{(R+P)^{d}}
  (R+P)^{d}=P^{\pi}\sum_{i=0}^{m-1}P^{i}(R^{d})^{i+1}+\sum_{i=0}^{m-1}(P^{d})^{i+1}R^{i}R^{\pi},
\end{align}
where $\max\{\ind(R),\ind(P)\}\leq m\leq \ind(R)+\ind(P)$. Combining $\eqref{Ppi(R+P)}$, $\eqref{(S+Q)^{d}}$ and $\eqref{(R+P)^{d}}$ we have
$$((S+Q)+P^{\pi}(R+P))^{d}=X_1+X_2,$$ where
$$X_1=(1-P^{\pi}(R+P)(R+P)^{d})(1+\sum_{j=1}^{n-1}P^{\pi}(R+P)^{j})[(S+Q)^d]^{j+1}$$
and
$$X_2=\sum_{j=0}^{n-1}P^{\pi}[(R+P)^{d}]^{j+1}(S+Q)^{j}(1-(S+Q)(S+Q)^{d}).$$
Since
\begin{align*}
    X_1=&(1-P^{\pi}(R+P)(P^{\pi}\sum_{i=0}^{m-1}P^{i}(R^{d})^{i+1}+\sum_{i=0}^{m-1}(P^{d})^{i+1}R^{i}R^{\pi}))\\
    &\quad\times(1+\sum_{j=1}^{n-1}P^{\pi}(R+P)^{j})(Q^{\pi}\sum_{h=0}^{l-1}Q^{h}(S^{d})^{h+1}+\sum_{h=0}^{l-1}(Q^{d})^{h+1}S^{h}S^{\pi})^{j+1}\\
    =&(1-P^{\pi}(R+P)\sum_{i=0}^{m-1}P^{i}(R^{d})^{i+1})\\
    &\quad\times(P^\pi+\sum_{j=1}^{n-1}P^\pi(R+P)^j)
    (Q^\pi\sum_{h=0}^{l-1}Q^h(S^d)^{h+j+1}-\sum_{k=1}^{j}(Q^d)^k(S^d)^{j+1-k}+\sum_{h=0}^{l-1}(Q^d)^{h+j+1}S^{h}S^\pi)\\
    =&(1-P^\pi RR^d-P^\pi\sum_{i=0}^{m-1}P^{i+1}(R^d)^{i+1})\\
    &\quad\times
    P^\pi\sum_{j=0}^{n-1}\sum_{k=0}^{j}P^kR^{j-k}
    (Q^\pi\sum_{h=0}^{l-1}Q^h(S^d)^{h+j+1}-\sum_{k=1}^{j}(Q^d)^k(S^d)^{j+1-k}+\sum_{h=0}^{l-1}(Q^d)^{h+j+1}S^{h}S^\pi)\\
=&P^\pi(\sum_{j=0}^{n-1}(\sum_{k=1}^{j}P^kR^{j-k}+R^{\pi}R^j)-\sum_{i=0}^{m-1}\sum_{j=0}^{n-1}P^{i+1}(R^d)^{i+1}R^j)\\
&\quad\times(Q^\pi\sum_{h=0}^{l-1}Q^h(S^d)^{h+j+1}-\sum_{k=1}^{j}(Q^d)^k(S^d)^{j+1-k}+\sum_{h=0}^{l-1}(Q^d)^{h+j+1}S^{h}S^\pi)\\
 \end{align*} and
 \begin{align*}
    X_2&=\sum_{j=0}^{n-1}P^{\pi}(P^{\pi}\sum_{i=0}^{m-1}P^{i}(R^{d})^{i+1}+\sum_{i=0}^{m-1}(P^{d})^{i+1}R^{i}R^{\pi})^{j+1}\\
    &\quad\times(S+Q)^{j}(1-(S+Q)(Q^{\pi}\sum_{h=0}^{l-1}Q^{h}(S^{d})^{h+1}+\sum_{h=0}^{l-1}(Q^{d})^{h+1}S^{h}S^{\pi}))\\
   =&P^{\pi}\sum_{j=0}^{n-1}(\sum_{i=0}^{m-1}P^{i}(R^{d})^{i+1})^{j+1}\\
    &\quad\times\sum_{k=0}^{j}Q^kS^{j-k}(1-SS^d-Q^\pi\sum_{h=0}^{l-1}Q^{h+1}(S^d)^{h+1}-\sum_{h=0}^{l-1}Q(Q^d)^{h+1}S^{h}S^\pi)\\
    =&P^\pi\sum_{j=0}^{n-1}\sum_{i=0}^{m-1}P^i(R^d)^{i+j+1}(\sum_{k=0}^{j}Q^kS^{j-k}S^\pi-Q^\pi\sum_{h=0}^{l-1}Q^{h+1+j}(S^d)^{h+1}-\sum_{h=0}^{l-1}Q^{j+1}(Q^d)^{h+1}S^{h}S^\pi),
 \end{align*}
 we finish the proof.
\end{proof}

\begin{lemma}\label{triangle}(\cite{Hartwig(1977)} and \cite{Meyer(1977)})
  Let $M=\left[\begin{array}{cc}
                   A & C \\
                   0 & D
                 \end{array}\right]$
     and let $N=\left[\begin{array}{cc}
                   D & 0 \\
                   C & A
                 \end{array}\right]\in \mathbb{C}^{n\times n}$,
  where $A$ and $D$ are square matrices. Then
  $$M^{d}=\left[\begin{array}{cc}
                   A^{d} & X \\
                   0 & D^{d}
                 \end{array}\right] ~~~\text{and}~~
  N^{d}=\left[\begin{array}{cc}
                   D^{d} & 0 \\
                   X & A^{d}
                 \end{array}\right],$$
   where $$
   X=\sum_{i=0}^{s-1}(A^{d})^{i+2}CD^{i}D^{\pi}+A^{\pi}\sum_{i=0}^{r-1}A^{i}C(D^{d})^{i+2}-A^{d}CD^{d},$$
 $r=\ind(A)$ and $s=\ind(D)$.
\end{lemma}

\begin{theorem}\label{pqrssum}
Let $P,Q,R$ and $S$ $\in\mathbb{C}^{n\times n}$ and let $N=P+Q+R+S$. If $PQ=QP=0,PS=SQ=QR=RP=0$ and $SP=SR=0$, then
  \begin{equation*}
 N^{d}= P^{d}+X+T^{d},
\end{equation*}
where
 \begin{align*}
 T=&(S+Q)+P^{\pi}(R+P),\\
   X=&\sum_{k=0}^{g-1}(P^{d})^{k+2}RT^{k}T^{\pi}-P^{d}RT^{d},\\
   T^{d}=&P^\pi(\sum_{j=0}^{n-1}(\sum_{k=1}^{j}P^kR^{j-k}+R^{\pi}R^j)-\sum_{i=0}^{m-1}\sum_{j=0}^{n-1}P^{i+1}(R^d)^{i+1}R^j)\\
&\quad\times(Q^\pi\sum_{h=0}^{l-1}Q^h(S^d)^{h+j+1}-\sum_{k=1}^{j}(Q^d)^k(S^d)^{j+1-k}+\sum_{h=0}^{l-1}(Q^d)^{h+j+1}S^{h}S^\pi)\\
&+P^\pi\sum_{j=0}^{n-1}\sum_{i=0}^{m-1}P^i(R^d)^{i+j+1}
(\sum_{k=0}^{j}Q^kS^{j-k}S^\pi-Q^\pi\sum_{h=0}^{l-1}Q^{h+1+j}(S^d)^{h+1}-\sum_{h=0}^{l-1}Q^{j+1}(Q^d)^{h+1}S^{h}S^\pi),
 \end{align*}
for
 \begin{align*}
 & \ind(T)=g,\\
 &\max\{\ind(S),\ind(Q)\}\leq l\leq \ind(S)+\ind(Q),\\
 &\max\{\ind(R),\ind(P)\}\leq m\leq \ind(R)+\ind(P),\\
 &\max\{\ind(S+Q),\ind(P^{\pi}(R+P))\}\leq n \leq \ind(S+Q)+\ind(P^{\pi}(R+P)).
 \end{align*}
\end{theorem}
\begin{proof}
Using Lemma \ref{M=P+Q+R+S:equivalentM,N}, we have
  \begin{equation*}
    \sigma(N)=\left[\begin{array}{cc}
       P^{e}P & P^{e}R \\
       SP^{e} & Q+P^{\pi}R+(S+P)P^{\pi}
     \end{array}\right]
  \end{equation*}
and so, by $SP=0$,
 \begin{equation*}
    \sigma(N)=\left[\begin{array}{cc}
       P^{e}P & P^{e}R \\
            0 & (S+Q)+P^{\pi}(R+P)
     \end{array}\right].
  \end{equation*}
Note that $T=(S+Q)+P^{\pi}(R+P)$.
Combining Lemma \ref{triangle} and Lemma \ref{SP=SR=0,2-2} we get
 \begin{equation*}
   \sigma^{d}(N)=\left[\begin{array}{cc}
           P^{d} & X \\
             0   & T^{d}
         \end{array}\right],
 \end{equation*}
where
 \begin{align*}
   X=&\sum_{k=0}^{g-1}(P^{d})^{k+2}\cdot P^{e}R\cdot T^{k}T^{\pi}
   +(P^{e}P)^{\pi}\sum_{k=0}^{f-1}(P^{e}P)^{k}\cdot P^{e}R\cdot (T^{d})^{k+2}
   -(P^{e}P)^{d}\cdot P^{e}R\cdot T^{d}\\
    =&\sum_{k=0}^{g-1}(P^{d})^{k+2}RT^{k}T^{\pi}-P^{d}RT^{d},\\
   T^{d}=&P^\pi(\sum_{j=0}^{n-1}(\sum_{k=1}^{j}P^kR^{j-k}+R^{\pi}R^j)-\sum_{i=0}^{m-1}\sum_{j=0}^{n-1}P^{i+1}(R^d)^{i+1}R^j)\\
&\quad\times(Q^\pi\sum_{h=0}^{l-1}Q^h(S^d)^{h+j+1}-\sum_{k=1}^{j}(Q^d)^k(S^d)^{j+1-k}+\sum_{h=0}^{l-1}(Q^d)^{h+j+1}S^{h}S^\pi)\\
&+P^\pi\sum_{j=0}^{n-1}\sum_{i=0}^{m-1}P^i(R^d)^{i+j+1}
(\sum_{k=0}^{j}Q^kS^{j-k}S^\pi-Q^\pi\sum_{h=0}^{l-1}Q^{h+1+j}(S^d)^{h+1}-\sum_{h=0}^{l-1}Q^{j+1}(Q^d)^{h+1}S^{h}S^\pi),
 \end{align*}
for
 \begin{align*}
 &\ind(P^{e}P)=f=1, \ind(T)=g,\\
 &\max\{\ind(S),\ind(Q)\}\leq l\leq \ind(S)+\ind(Q),\\
 &max\{\ind(R),\ind(P)\}\leq m\leq \ind(R)+\ind(P),\\
 &max\{\ind(S+Q),\ind(P^{\pi}(R+P))\}\leq n \leq \ind(S+Q)+\ind(P^{\pi}(R+P)).
 \end{align*}
Considering Lemma \ref{M=P+Q+R+S:equivalentM,N}, we obtain
\begin{equation*}
 N^{d}= P^{d}+X+T^{d}
\end{equation*}
as desired.
\end{proof}

The following conditions have been considered in \cite[Theorem 4.1]{Chen2008}.
As a direct corollary of Theorem \ref{pqrssum}, the new representation of the Drazin inverse of the sum
$P+Q+R+S$.
\begin{corollary}\label{corPQRS-3.7}{\rm \cite[Theorem 4.1]{Chen2008}} Let $P,Q,R$ and $S$ $\in\mathbb{C}^{n\times n}$ and let $N=P+Q+R+S$.
If $PQ=QP=0,PS=SQ=QR=RP=0$, $SP=SR=0$ and $R^d=S^d=0$, then
  \begin{equation*}
 N^{d}= P^{d}+X+T^{d},
\end{equation*}
where $T=(S+Q)+P^{\pi}(R+P)$,
 \begin{align*}
   X=&\sum_{k=0}^{g-1}(P^{d})^{k+2}RT^{k}T^{\pi}-P^{d}\sum_{j=0}^{n-1}R^{j+1}\sum_{h=0}^{l-1}(Q^d)^{h+j+1}S^{h},\\
   T^{d}=&P^\pi\sum_{j=0}^{n-1}\sum_{k=0}^{j}P^kR^{j-k}\sum_{h=0}^{l-1}(Q^d)^{h+j+1}S^{h},
 \end{align*}
for
 \begin{align*}
 & \ind(T)=g,\\
 &\max\{\ind(S),\ind(Q)\}\leq l\leq \ind(S)+\ind(Q),\\
 &\max\{\ind(S+Q),\ind(P^{\pi}(R+P))\}\leq n \leq \ind(S+Q)+\ind(P^{\pi}(R+P)).
 \end{align*}
\end{corollary}
\begin{proof}
If we combine $R^d=S^d=0$ and Theorem \ref{pqrssum}, then the expression for $T^d$ can be simplified as follows:
  \begin{align*}
   T^{d}=&P^\pi\sum_{j=0}^{n-1}\sum_{k=0}^{j}P^kR^{j-k}\sum_{h=0}^{l-1}(Q^d)^{h+j+1}S^{h}.
  \end{align*}
  Hence,
  \begin{align*}
   RT^{d}=&\sum_{j=0}^{n-1}R^{j+1}\sum_{h=0}^{l-1}(Q^d)^{h+j+1}S^{h}.
  \end{align*}
  The rest of the proof is obvious.
\end{proof}

\begin{remark}
We conclude this section with a remark. Using a way similar to
Theorem \ref{pqrssum} we can give an expression of the Drazin
inverse of the sum $P+Q+R+S$ under the following condition:
 $$PQ=QP=0,PS=SQ=QR=RP=0, RS=QS=0,$$
which generalized the condition of \cite[Theorem 4.2]{Chen2008}.
\end{remark}

\section{Applications}
In this section, as an application of previous results, we
obtain representations for the Drazin inverse of a $2 \times 2$
block matrix. Throughout this section let
$$M=\left[\begin{array}{cc}
                   A & B \\
                   C & D
                 \end{array}\right]\quad
                 {\rm and}\quad Z=D-CA^dB,$$
where $A\in\mathbb{C}^{m\times m}$, $B\in\mathbb{C}^{m\times n}$, $C\in\mathbb{C}^{n\times m}$ and $D\in\mathbb{C}^{n\times n}$.

First, we state one auxiliary result.

\begin{lemma}\label{miaoDrazin}{\rm \cite{Miao}}
If $CA^\pi=0$, $A^\pi B=0$ and the generalized Schur complement $Z=D-CA^dB=0$, then
  $$M^{d}=\left[\begin{array}{c}
                   I \\
                   CA^{d}
                 \end{array}\right]
                 [(AW)^d]^2A
                 \left[\begin{array}{cc}
                   I & A^{d}B
                 \end{array}\right],$$
where $W=AA^d+A^dBCA^d$.
\end{lemma}

Under conditions $BZ=0$, $ZC=0$ and $A^\pi BC=0$, we obtain the
following expression for $M^d$ applying Corollary
\ref{corPQRS-3.7} or Theorem \ref{pqrssum}.

\begin{corollary}\label{corPQRS-3.9} If $BZ=0$, $ZC=0$ and $A^\pi BC=0$, then
$$M^d=P^d+\sum_{k=0}^{g-1}(P^{d})^{k+2}\left[\begin{array}{cc}
                   0& 0 \\
                   CA^\pi & 0
                 \end{array}\right]T^{k}T^{\pi}+T^{d},$$
                 where $W=AA^d+A^dBCA^d$, $$P=\left[\begin{array}{cc}
                   A^2A^d & AA^dB \\
                   CAA^d & CA^dB
                 \end{array}\right],\qquad
                 P^{d}=\left[\begin{array}{c}
                   I \\
                   CA^{d}
                 \end{array}\right]
                 [(AW)^d]^2A
                 \left[\begin{array}{cc}
                   I & A^{d}B
                 \end{array}\right],$$
                 $$T=\left[\begin{array}{cc}
                   AA^\pi & A^\pi B \\
                   0 & Z
                 \end{array}\right]+P^{\pi}\left[\begin{array}{cc}
                   A^2A^d & AA^dB \\
                   C & CA^dB
                 \end{array}\right],$$
 \begin{align*}
   T^{d}=&P^\pi\sum_{j=0}^{n-1}P^j\left[\begin{array}{cc}
                   0 & 0 \\
                   0 & (Z^d)^{j+1}
                 \end{array}\right],
 \end{align*}
for $\ind(T)=g$ and $\ind(P)=n$.
\end{corollary}

\begin{proof} Suppose that $M=P+Q+R+S$, where $$P=\left[\begin{array}{cc}
                   A^2A^d & AA^dB \\
                   CAA^d & CA^dB
                 \end{array}\right],\quad
                 Q=\left[\begin{array}{cc}
                   AA^\pi & 0 \\
                   0 & Z
                 \end{array}\right],\quad
                 R=\left[\begin{array}{cc}
                   0& 0 \\
                   CA^\pi & 0
                 \end{array}\right],\quad
S=\left[\begin{array}{cc}
                   0 & A^\pi B \\
                   0 & 0
                 \end{array}\right].$$
Since $BZ=0=ZC=A^\pi BC$, we get $PQ=QP=0,PS=SQ=QR=RP=0$ and $SP=SR=0$.

If $a_p=A^2A^d$, $b_p=AA^dB$, $c_p=CAA^d$ and $d_p=CA^dB$, then $c_pa_p^\pi=CAA^dA^\pi=0$,
$a_p^\pi b_p=A^\pi AA^dB=0$ and $z_p=d_p-c_pa_pb_p=0$. Applying Lemma \ref{miaoDrazin}, we obtain
$$P^{d}=\left[\begin{array}{c}
                   I \\
                   CA^{d}
                 \end{array}\right]
                 [(AW)^d]^2A
                 \left[\begin{array}{cc}
                   I & A^{d}B
                 \end{array}\right].$$

Using Lemma \ref{triangle}, we observe that $R^d=S^d=0$ and
$$Q^d=\left[\begin{array}{cc}
                   0 & 0 \\
                   0 & Z^d
                 \end{array}\right].$$

By Corollary \ref{corPQRS-3.7}, we obtain the expression for the Drazin inverse of $M$:
\begin{equation*}
 M^{d}= P^{d}+X+T^{d},
\end{equation*}
where $$T=\left[\begin{array}{cc}
                   AA^\pi & A^\pi B \\
                   0 & Z
                 \end{array}\right]+P^{\pi}\left[\begin{array}{cc}
                   A^2A^d & AA^dB \\
                   C & CA^dB
                 \end{array}\right],$$
 \begin{align*}
 T^{d}=&P^\pi\sum_{j=0}^{n-1}P^j\left[\begin{array}{cc}
                   0 & 0 \\
                   0 & (Z^d)^{j+1}
                 \end{array}\right],\\
   X=&\sum_{k=0}^{g-1}(P^{d})^{k+2}\left[\begin{array}{cc}
                   0& 0 \\
                   CA^\pi & 0
                 \end{array}\right]T^{k}T^{\pi},\\
 \end{align*}
for $\ind(T)=g$ and $\ind(P)=n$.
\end{proof}

If we assume that $Z=0$ in Corollary \ref{corPQRS-3.9}, we obtain
the next result which parts (i) and (ii) are
\cite[Corollary 2.10 (i)]{MDjamc2} and \cite[Corollary
2.1(ii)]{MDijcmEx} for cases of Drazin inverses, respectively.

\begin{corollary} Let $W$, $P$ and $P^d$ be defined as in Corollary \ref{corPQRS-3.9}. If
\begin{itemize}
\item[\rm (i)] $Z=0$ and $A^\pi BC=0$, then
$$M^d=P^d+\sum_{k=0}^{g-1}(P^{d})^{k+2}\left[\begin{array}{cc}
                   0& 0 \\
                   CA^\pi & 0
                 \end{array}\right]\left[\begin{array}{cc}
                   AA^\pi & A^\pi B \\
                   0 & 0
                 \end{array}\right]^{k},$$
                 where $\ind(A)=g+1$. 

\item[\rm (ii)] $Z=0$ and $A^\pi B=0$, then
$$M^d=P^d+\sum_{k=0}^{g-1}(P^{d})^{k+2}\left[\begin{array}{cc}
                   0& 0 \\
                   CA^\pi A^k& 0
                 \end{array}\right],$$
                 where $\ind(A)=g$.
\end{itemize}
\end{corollary}

Now, we consider the formula for the Drazin inverse of $M$ in the
case that $BZ=0$, $ZCA^d=0$, $A^\pi BC=0$ and $BCA^\pi=0$.

\begin{corollary}\label{corPQRS-3.10} If $BZ=0$, $ZCA^d=0$, $A^\pi BC=0$ and $BCA^\pi=0$, then
$$M^d=\left[\begin{array}{cc}
                   0 & 0 \\
                   0 & Z^d
                 \end{array}\right]+Q^d+\sum_{k=0}^{g-1}\left[\begin{array}{cc}
                   0& 0 \\
                   (Z^d)^{k+2}C & 0
                 \end{array}\right]
                 \left[\begin{array}{cc}
                   A & B \\
                   Z^\pi C& Z^\pi D
                 \end{array}\right]^{k}(I-\left[\begin{array}{cc}
                   A & B \\
                   Z^\pi C& Z^\pi D
                 \end{array}\right]Q^d)^{\pi},$$
                 where $\ind(\left[\begin{array}{cc}
                   A & B \\
                   Z^\pi C& Z^\pi D
                 \end{array}\right])=g$, $W=AA^d+A^dBCA^d$ and $$Q^{d}=\left[\begin{array}{c}
                   I \\
                   CA^{d}
                 \end{array}\right]
                 [(AW)^d]^2A
                 \left[\begin{array}{cc}
                   I & A^{d}B
                 \end{array}\right].$$
\end{corollary}

\begin{proof} Let $M=P+Q+R+S$, where
$$P=\left[\begin{array}{cc}
                   0 & 0 \\
                   0 & Z
                 \end{array}\right],\quad
                 Q=\left[\begin{array}{cc}
                   A^2A^d & AA^dB \\
                   CAA^d & CA^dB
                 \end{array}\right],\quad
                 R=\left[\begin{array}{cc}
                   0& 0 \\
                   CA^\pi & 0
                 \end{array}\right],\quad
S=\left[\begin{array}{cc}
                   AA^\pi & A^\pi B \\
                   0 & 0
                 \end{array}\right].$$
We check that $R^d=S^d=0$, $$P^d=\left[\begin{array}{cc}
                   0 & 0 \\
                   0 & Z^d
                 \end{array}\right]\quad{\rm and}\quad Q^{d}=\left[\begin{array}{c}
                   I \\
                   CA^{d}
                 \end{array}\right]
                 [(AW)^d]^2A
                 \left[\begin{array}{cc}
                   I & A^{d}B
                 \end{array}\right].$$
Applying Corollary \ref{corPQRS-3.7}, we get
\begin{equation*}
 M^{d}= P^{d}+X+T^{d},
\end{equation*}
where $$T=\left[\begin{array}{cc}
                   A & B \\
                   CAA^d & CA^dB
                 \end{array}\right]+\left[\begin{array}{cc}
                   I & 0 \\
                   0 & Z^\pi
                 \end{array}\right]
                 \left[\begin{array}{cc}
                   0&0\\
                   CA^\pi& Z
                 \end{array}\right]
                 =\left[\begin{array}{cc}
                   A & B \\
                   Z^\pi C& Z^\pi D
                 \end{array}\right],$$
$$T^{d}=Q^d,\qquad X=\sum_{k=0}^{g-1}\left[\begin{array}{cc}
                   0& 0 \\
                   (Z^d)^{k+2}C & 0
                 \end{array}\right]T^{k}T^{\pi},$$
for $\ind(T)=g$.
\end{proof}

The following result extends  \cite[Corollary 4.2]{HLW}, which includes conditions
$CA^\pi A=0$, $CA^\pi B=0$ and the generalized Schur complement $Z=0$.

\begin{corollary}\label{corPQRS-3.11} If $BZ=0$, $ZC=0$, $CA^\pi A=0$ and $CA^\pi B=0$, then
\begin{eqnarray*}M^d&=&\left[\begin{array}{cc}
                   0 & 0 \\
                   0 & Z^d
                 \end{array}\right]+Q^d+(Q^{d})^{2}\left[\begin{array}{cc}
                   0& 0 \\
                   CA^\pi & 0
                 \end{array}\right]\\&+&
                 \sum_{j=1}^{n-1}\left[\begin{array}{cc}
                   0 & A^{j-1}A^\pi B \\
                   0 & 0
                 \end{array}\right]
                 \big((Q^{d})^{j+1}+(Q^{d})^{j+2}\left[\begin{array}{cc}
                   0& 0 \\
                   CA^\pi & 0
                 \end{array}\right]\big),\end{eqnarray*}
                 where $\ind(A)=n-1$, $W=AA^d+A^dBCA^d$ and $$Q^{d}=\left[\begin{array}{c}
                   I \\
                   CA^{d}
                 \end{array}\right]
                 [(AW)^d]^2A
                 \left[\begin{array}{cc}
                   I & A^{d}B
                 \end{array}\right].$$
\end{corollary}

\begin{proof} If we write $M=P+Q+R+S$, for $$P=\left[\begin{array}{cc}
                   0 & 0 \\
                   0 & Z
                 \end{array}\right],\quad
                 Q=\left[\begin{array}{cc}
                   A^2A^d & AA^dB \\
                   CAA^d & CA^dB
                 \end{array}\right],\quad
                 R=\left[\begin{array}{cc}
                   AA^\pi& A^\pi B \\
                   0& 0
                 \end{array}\right],\quad
S=\left[\begin{array}{cc}
                   0 & 0 \\
                   CA^\pi & 0
                 \end{array}\right],$$
we verify this result as Corollary \ref{corPQRS-3.10}.
\end{proof}

As a consequence of Corollary \ref{corPQRS-3.10} and Corollary
\ref{corPQRS-3.11}, we get the next result which parts (i) and
(ii) are \cite[Corollary 2.2 (ii)]{MDijcmEx}
and \cite[Corollary 2.3 (ii)]{MDjrepgdAMC1} for cases of Drazin inverses, respectively.

\begin{corollary} Let $W$ and $Q^{d}$ be defined as in Corollary \ref{corPQRS-3.10}. If
\begin{itemize}
\item[\rm (i)] $Z=0$ and $CA^\pi=0$, then
$$M^d=Q^d+\sum_{j=1}^{n-1}\left[\begin{array}{cc}
                   0 & A^{j-1}A^\pi B \\
                   0 & 0
                 \end{array}\right]
                (Q^{d})^{j+1},$$
                 where $\ind(A)=n-1$.

\item[\rm (ii)] $Z=0$, $A^\pi B=0$ and $BCA^\pi=0$, then
$M^d=Q^d$.
\end{itemize}
\end{corollary}

The assumptions $BC=0$ and $BD=0$ imply the following result which
recovers \cite[Theorem 3.2 (i)]{MZCannalsFunc} and
\cite[Theorem 5.3]{DjS} for the case of Drazin inverses.

\begin{corollary} If $BC=0$ and $BD=0$, then
\begin{eqnarray*}M^d&=&\left[\begin{array}{cc}
                   0 & 0 \\
                   -D^dCA^d & D^d-D^dC(A^d)^2B
                 \end{array}\right]+T^d\\&+&\sum_{k=0}^{g-1}\left[\begin{array}{cc}
                   0& 0 \\
                   (D^d)^{k+2}C & 0
                 \end{array}\right]
                 \left[\begin{array}{cc}
                   A & B \\
                   D^\pi C& D^\pi D
                 \end{array}\right]^{k}(I-\left[\begin{array}{cc}
                   A & B \\
                   D^\pi C& D^\pi D
                 \end{array}\right]T^d)^{\pi},\end{eqnarray*}
                 where $\ind(\left[\begin{array}{cc}
                   A & B \\
                   D^\pi C& D^\pi D
                 \end{array}\right])=g$ and $$T^{d}=\left[\begin{array}{cc}
                   A^{d} & (A^{d})^2B \\
                   0     &  0
                 \end{array}\right]
                 +\sum_{j=1}^{n-1}
                 \left[\begin{array}{cc}
                   0 & 0\\ D^\pi D^{j-1}C(A^{d})^{j+1} & D^\pi D^{j-1}C(A^{d})^{j+2}B
                 \end{array}\right].$$
\end{corollary}

\begin{proof} To show this result, we use $M=P+Q+R+S$ for
$$P=\left[\begin{array}{cc}
                   0 & 0 \\
                   0 & D
                 \end{array}\right],\quad
                 Q=\left[\begin{array}{cc}
                   A & 0 \\
                   0 & 0
                 \end{array}\right],\quad
                 R=\left[\begin{array}{cc}
                   0& 0 \\
                   C & 0
                 \end{array}\right],\quad
S=\left[\begin{array}{cc}
                   0 & B \\
                   0 & 0
                 \end{array}\right]$$
as desired.
\end{proof}

\bibliographystyle{amsplain}

\end{document}